\documentclass[10pt, leqno]{article}

\usepackage{amsmath,amssymb,amsthm, epsfig}

\usepackage{hyperref}

\usepackage{amsmath}

\usepackage{amssymb}

\usepackage{hyperref}

\usepackage{color}

\usepackage{ulem}

\usepackage{epsfig}

\usepackage{dsfont}

\usepackage{mathrsfs}

\usepackage{mathtools}

\usepackage{stackrel}

\title{Sharp regularity estimates for a singular inhomogeneous $(m, p)$-Laplacian equation}
\author{\it by \smallskip \\
P\^{e}dra D. S. Andrade,\footnote{\noindent \textsc{P\^{e}dra Daricl\'{e}a Santos Andrade}.
 Instituto Superior T\'{e}cnico, Universidade de Lisboa, 1049-001, Lisboa, Portugal.
\texttt{E-mail address: pedra.andrade@tecnico.ulisboa.pt}
}
\smallskip \qquad Jo\~{a}o Vitor da Silva,\footnote{\noindent \textsc{Jo\~ao Vitor da Silva}.
Universidade Estadual de Campinas - UNICAMP. Department of Mathematics. Campinas - SP, Brazil.
\texttt{E-mail address: jdasilva@unicamp.br}
}
\smallskip \qquad Giane C. Rampasso\footnote{\noindent \textsc{Giane Casari Rampasso}.
Universidade Federal de Itajubá - UNIFEI. Instituto de Matemática e Computação, Campus Prof. José Rodrigues Seabra. Itajubá - MG, Brazil.
\texttt{E-mail address: gianecr@unifei.edu.br}
}
\smallskip \\
\quad $\&$ \quad
\smallskip \\
Makson S. Santos\footnote{\noindent \textsc{Makson Sales Santos}.
 Instituto Superior T\'{e}cnico, Universidade de Lisboa, 1049-001, Lisboa, Portugal.
\texttt{E-mail address: makson.santos@tecnico.ulisboa.pt}
}}

\newlength{\hchng}
\newlength{\vchng}
\def \Div {\mathrm{div}}
\def \dist {\mathrm{dist}}

\def \loc {\mathrm{loc}}

\newcommand{\defeq}{\mathrel{\mathop:}=}

\newtheorem{theorem}{Theorem}[section]

\newtheorem{lemma}[theorem]{Lemma}

\newtheorem{proposition}[theorem]{Proposition}

\theoremstyle{definition}

\newtheorem{definition}[theorem]{Definition}

\theoremstyle{remark}

\newtheorem{remark}[theorem]{Remark}

\newtheorem{Assumption}{A}

\numberwithin{equation}{section}

\newcommand{\intav}[1]{\mathchoice {\mathop{\vrule width 6pt height 3 pt depth  -2.5pt
\kern -8pt \intop}\nolimits_{\kern -6pt#1}} {\mathop{\vrule width
5pt height 3  pt depth -2.6pt \kern -6pt \intop}\nolimits_{#1}}
{\mathop{\vrule width 5pt height 3 pt depth -2.6pt \kern -6pt
\intop}\nolimits_{#1}} {\mathop{\vrule width 5pt height 3 pt depth
-2.6pt \kern -6pt \intop}\nolimits_{#1}}}

\begin{document}
\maketitle

\begin{abstract}

In this paper, we investigate a class of doubly nonlinear evolutions PDEs. We establish sharp regularity for the solutions in H\"{o}lder spaces. The proof is based on the geometric tangential method and intrinsic scaling technique. Our findings extend and recover the results in the context of the classical evolution PDEs with singular signature via a unified treatment in the slow, normal and fast diffusion regimes. In addition, we provide some applications to certain nonlinear evolution models, which may have their own mathematical interest.

\bigskip


\bigskip


\end{abstract}



\section{Introduction}

We examine the regularity for bounded weak solutions to doubly nonlinear evolution PDEs of singular type, namely the inhomogeneous
$(m,p)$-Laplacian equation:
\begin{equation}\label{1.1}
\mathcal{Q}_{m, p} [u]:= \frac{\partial u}{\partial t}-\Div(m|u|^{m-1}|\nabla u|^{p-2}\nabla u) = f(x, t) \quad \text{in} \quad Q^-_1,
\end{equation}
where $m\ge 1$, $\max\left\{1, \frac{2n}{n+2}\right\}< p \leq  2$, $f$ belongs to an appropriate Lebesgue space with mixed norms (cf. \cite{BP}), and $Q^-_1$ is an intrinsic cylinder to be defined later. We produce new and sharp estimates for the weak solutions to \eqref{1.1} in H\"older spaces. It is convenient to notice that for such range of $m$ and $p$ our model generalizes both the well-known porous medium model, when $p=2$, and the evolution $p$-Laplacian operator, when $m=1$. Particularly, in the case $p=2$ and $m=1$ we recover the classical Heat equation. For further details, we refer the reader to \cite{AMU20, daST17, Diehl21, K08, TU14}.

The $(m,p)$-Laplacian operator appears in many physics phenomena, ranging from the study of turbulent filtration of a gas, fluids through porous media, in theoretical glaciology, plasma physics, image processing, and groundwater problems to the motion of viscous fluids, see \cite{AMS04, Iva97, Iva00} and \cite{Kalash87}.  For instance, in the setting $n = 1$, $p \in [3/2, 2)$ and $m > p - 1$, equation \eqref{1.1} arises in the study of a turbulent flow of a gas in a one-dimensional porous media.  Moreover, in the three-dimensional case, similar equations (to not say more general) of the form
\[
u_t - \mbox{div}\{|u|^{m}|\nabla u - c_0|u|^ke_n|^{p-2}(\nabla u - c_0|u|^ke_n)\} = 0,
\]
for suitable exponents $m, p, k$, and $e_n := (0,0,1)$ appear when studying porous medium in turbulent regimes. We refer the reader to \cite{AMS04, DT94, Iva97, Iva00, Leibenson45} and the references therein.

Besides the applications mentioned above, the $(m,p)$-Laplacian operator drew the attention of many authors due to others properties such as its nonlinear structure, and the switch of the speed of propagation that depends on the quantity $m+p$. To illustrate the later phenomenon, let us consider a given solution $u \in C^2(Q^-_1)$ to \eqref{1.1}. We may rewrite the equation \eqref{1.1} as follows:
\begin{equation*}
  \frac{\partial u}{\partial t}-\mathrm{E}(u,\nabla u )\Delta u  - \mathrm{F}(u, \nabla u)- \mathrm{G}(u, \nabla u)\Delta^N_{\infty}\,u = f(x, t) \quad \text{in} \quad Q^-_1,
\end{equation*}
where
\[
\left\{
\begin{array}{rcl}
  \mathrm{E}(u,\nabla u ) & \defeq & m |u|^{m-1}|\nabla u|^{p-2} \\
  \mathrm{F}(u, \nabla u) & \defeq & m (m-1)|u|^{m-3}u| \nabla u|^{p} \\
  \mathrm{G}(u, \nabla u) & \defeq & m(p-2)|u|^{m-1}|\nabla u|^{p-2},
\end{array}
\right.
\]
and
\[
\displaystyle \Delta_{\infty}\,u \defeq \frac{1}{|\nabla u|^2}\sum_{i, j=1}^{n} \frac{\partial^2 u}{\partial x_i \partial x_i}\frac{\partial u}{\partial x_i}\frac{\partial u}{\partial x_j} = \left\langle D^2 u \frac{\nabla u}{|\nabla u|}, \frac{\nabla u}{|\nabla u|} \right\rangle,
\]
is the normalized $\infty$-Laplacian operator. It is easy to see that
\begin{equation}\label{Diffusion_Coef}
 \lim_{\varrho \rightarrow 0} \mathrm{E}(\varrho u, \varrho \nabla u )=
 \left\{
 \begin{array}{ccl}
 0 & \quad m+p > 3 \vspace{0.1cm}  \\
 E(u,\nabla u ) & \quad m+p = 3 \vspace{0.1cm} \\
 + \infty  & \quad  m+ p <3 \\
 \end{array}
 \right.
\end{equation}
Hence, in view of \eqref{Diffusion_Coef}, we can classify the diffusion regime of the $(m, p)$-Laplacian operator as follows:

\begin{table}[h]
\centering
 \begin{tabular}{c|c|c}
{\it PDEs of slow diffusion } & {\it PDEs of normal diffusion }  & {\it PDEs of fast diffusion } \\
\hline
 $m+p>3$ & $m+p=3$ & $m+p<3$  \\
\end{tabular}
\end{table}
\noindent Note that the degenerate/singular feature of the equation \eqref{1.1} is given by the nonlinear term
\[
\mathrm{E}(u,\nabla u )  \defeq  m| u|^{m-1}| \nabla u |^{p-2},
\]
hence, we are not able to use directly the standard regularity results available for classical models, as in \cite{K08}.

This class of equations has been extensively studied over the years. A broad range of development, including existence and uniqueness of weak solutions were established in \cite{IvaJag00, Stur17, Stur17-2}, under a suitable integrability conditions on the right-hand side (see, for instance \eqref{w-cc} below). The boundedness property for the weak solutions of \eqref{1.1} is the subject of \cite{Iva00} and \cite{Stur18}. For the $C^{0,\alpha}$-regularity (for some $\alpha \in (0,1)$) and the local behaviour of the weak solutions, we refer the reader to \cite{BS99, FSV14, Iva89, Iva91, Iva97, Iva95, Iva98, Iva00, PV93, Surn14} and \cite{Vesp22}.

Concerning the optimal regularity for the $(m,p)$-Laplacian equation, we stress the works of Ara\'{u}jo in \cite{Ara20} and Bezerra J\'{u}nior \textit{et al}. in \cite{BJdaSR22}, both of them studied this problem under the assumption $m >1$ and $p>2$, { \it i.e.}, $m+p-2\geq1$. More precisely, in \cite{Ara20} the author proved that any bounded weak solution of \eqref{1.1} belongs to $C^{\alpha, \frac{\alpha}{\theta}}_{\text{loc}}(Q_1^-)$, with
\begin{equation}\label{AraujoExponent}
  \alpha \defeq \min\left\{\frac{\alpha^{-}_0(p-1)}{m+p-2},\frac{(pq-n)r-pq}{q[(r-1)(m+p-2)+1]}\right \}.
\end{equation}
The notations $\alpha_0 \in (0, 1]$ denotes the optimal H\"{o}lder exponent to the homogeneous case and   $\alpha^{-}_0$ stands for any number $s$ such that $0< s <\alpha_0$. Moreover
\begin{equation}\label{Intr_Scalling}
  \theta \defeq p -\alpha(m+p-2)\Big(1-\frac{1}{m+p-2}\Big),
\end{equation}
 under the compatibility assumptions
\[
\frac{1}{r}+\frac{n}{pq}<1 \quad \text{and} \quad \frac{3}{r}+\frac{n}{q}>2.
\]
Furthermore, the authors in \cite{BJdaSR22} studied the case $m\geq1$ and $p\geq2$, under a slightly different condition
\[
\frac{1}{r}+\frac{n}{pq}<1 < \frac{2}{r}+\frac{n}{q},
\]
which implies
$$
\frac{1}{r}+\frac{n}{pq}<1 \quad \text{and} \quad  \frac{3}{r}+m\left(1-\frac{1}{r}\right)+\frac{n}{q}>2 \quad (\text{for} \quad q, r>1).
$$
In \cite{BJdaSR22}, they showed that bounded weak solutions to \eqref{1.1} are locally of class $C^{\alpha, \frac{\alpha}{\theta}}$, where $\alpha$ is the sharp exponent
\[
\alpha \defeq \min\left\{\max\left\{\frac{p\alpha^{-}_0}{\bar{p}+\alpha^{-}_0(m+p-3)}, \frac{2\alpha^{-}_0(p-1)}{\bar{p}(m+p-2)}\right\}, \frac{(pq-n)r-pq}{q[(r-1)(m+p-2)+1]}\right\},
\]
and
\begin{equation}\label{p(m)}
\bar{p} \defeq
\left\{
\begin{array}{rcl}
  2 & \mbox{if} & m=1 \\
  p & \mbox{if} & m>1.
\end{array}
\right.
\end{equation}

In sequel, we summarize the sharp H\"older regularity results obtained in the literature for the problems related to \eqref{1.1} in the table below:

\begin{table}[h]
\centering
\resizebox{\textwidth}{!}{
 \begin{tabular}{c|c|c|c}
{\it Model PDE } & {\it Compatibility condition }  & {\it H\"{o}lder regularity exponent} & \textit{Reference} \\
\hline
 $\mathcal{Q}_{m, p} [u]=f$ & $\frac{1}{r}+\frac{n}{pq}<1<\frac{2}{r}+\frac{n}{q}$ & $\alpha=\min\left\{\max\left\{\frac{p\alpha^{-}_0}{\bar{p}+\alpha^{-}_0(m+p-3)}, \frac{2\alpha^{-}_0(p-1)}{\bar{p}(m+p-2)}\right\}, \frac{(pq-n)r-pq}{q[(r-1)(m+p-2)+1]}\right\}$ &  \cite{BJdaSR22} \\
 \hline
 $\mathcal{Q}_{m, p} [u]=f $ & $\frac{1}{r}+\frac{n}{pq}<1 \,\, \text{and} \,\, \frac{3}{r}+\frac{n}{q}>2$ & $\alpha=\min\Big\{\frac{\alpha^{-}_0(p-1)}{m+p-2},\frac{(pq-n)r-pq}{q[(r-1)(m+p-2)+1]}\Big\}$ & \cite{Ara20} \\
\hline
$\frac{\partial u}{\partial t}-\Delta_{p} u=f $ & $\frac{1}{r}+\frac{n}{pq}<1< \frac{n}{q} + \frac{2}{r} $ & $\alpha=\frac{(pq-n)r-pq}{q[(p-1)r-(p-2)]}$ & \cite{TU14} \\
\hline
$\frac{\partial u}{\partial t}-\Delta u^{m}=f$ & $\frac{1}{r}+\frac{n}{2q} <1$ & $\alpha=\min\Big\{\frac{2\alpha^{-}_0}{2+(m-1)\alpha_0},\frac{(2q-n)r-2q}{q[mr-(m-1)]}\Big\}$ & \cite{Diehl21} \\
\hline 
$\frac{\partial u}{\partial t}-\Delta u^{m}=f$&$\frac{1}{r}+\frac{n}{2q}<1$&$\alpha=\min\Big\{\frac{\alpha^{-}_0}{m},\frac{[(2q-n)r-2q]}{q[mr-(m-1)]}\Big\}$& \cite{AMU20}\\
\hline
$\frac{\partial u}{\partial t}-\Delta u=f$ & $1< \frac{n}{q} + \frac{2}{r}< 2$ & $\alpha=2-\big(\frac{2}{r}+\frac{n}{q}\big)$& \cite{daST17},\cite{K08} \\
\end{tabular}}
\end{table}

To the best of our knowledge, there is no reference in the literature regarding the optimal regularity for the solutions to \eqref{1.1}, in the singular case $\max\left\{1, \frac{2n}{n+2}\right\}<p< 2$, even for $m+p-2 \geq 1$. Moreover, there is no mention of regularity theory for the {\it fast diffusion} case $m+p-1 < 2$. In this paper, we deal with these scenarios. More precisely, arguing by intrinsic scaling and approximation methods (see for instance \cite{AdaSRT, Ara20, JVSilva19, DOS18, TU14, TU21} and \cite{U08}), we derive the optimal regularity estimates that depend only on the structural and universal parameters of the problem, for the solution of the singular case $\max\left\{1, \frac{2n}{n+2}\right\}<p< 2$, in both regimes $m+p-1\leq 2$ and $m+p-1 > 2$. Our findings extends the  sharp regularity results of Ara\'{u}jo et al. \cite[Theorem 6]{AMU20}, Ara\'{u}jo \cite[Theorem 1.1]{Ara20}, Diehl \cite[Theorem 2.5]{Diehl21} and Teixeira-Urbano \cite[Theorem 3.4]{TU14}.

The remainder of this paper is organized as follows: In Section 2 we present
our main results; we also gather a few facts used throughout the paper and
detail our main assumptions. In Section 3 we state some auxiliary results used throughout the paper. Section 4 is devoted to geometric iterations that leads to the proof the desired H\"older estimates in Section 5. Finally, in the Section 6 and Section 7 we put forward some applications of our main result.

\section{Assumptions and main result}\label{Assump}

In this section, we detail the assumptions and the main result of this manuscript. We start with the assumption on the source term $f$. We define the {\it parabolic cylinders} by
\[
  Q^{-}_{\rho}(x_0, t_0)  = B_{\rho}(x_0) \times \left(t_0-\rho^{\theta}, t_0\right],
\]
where $\theta$ is defined in \eqref{Int-Scal} below.

\begin{Assumption}[{\bf Regularity of the source term}]\label{assump_f}
We will assume that $f \in L^{q, r}(Q^{-}_1) = L^r(-1, 0; L^q(B_1))$, {\it i.e.}, a Banach space endowed with the mixed norm:
\[
  \|f\|_{L^{q, r}(Q^{-}_1)} \defeq  \left(\int_{-1}^{0} \|f(\cdot, t)\|^r_{L^q(B_1)} dt\right)^{\frac{1}{r}} = \left(\int_{-1}^{0} \left(\int_{B_1}|f(x, t)|^qdx\right)^{\frac{r}{q}}dt\right)^{\frac{1}{r}}.
\]
\end{Assumption}

Next, we introduce the compatibility condition for the exponents $q$ and $r$.

\begin{Assumption}[{\bf Weaker compatibility condition}]\label{assump_w-cc} We suppose that the following  condition is satisfied:
\begin{equation}\label{w-cc}
\frac{1}{r}+\frac{n}{pq}<1 < \frac{2}{r}+\frac{n}{q}.
\end{equation}
\end{Assumption}

Because of the range of the parameters, the assumption $A$\ref{assump_w-cc} implies the stronger compatibility conditions
\begin{equation}\label{cc}
\frac{1}{r}+\frac{n}{pq}<1 \quad \mbox{and} \quad  \frac{3}{r}+m\left(1-\frac{1}{r}\right)+\frac{n}{q}>2 \quad (\mbox{for} \quad q, r>1).
\end{equation}

Also, it is worth noticing that assumption $A$\ref{assump_w-cc} provides the minimum requirements to the existence of bounded weak solutions to \eqref{1.1} that enjoy a H\"{o}lder modulus of continuity, and allow the use of the Caccioppoli's estimates. See \cite[Ch.2, \S 1]{D93} and \cite[Proposition 3.1]{FS08}.

As we mentioned before, we are particularly interested in obtaining the optimal  H\"older regularity estimates to \eqref{1.1} under the regimes:
\[
m+p-2\geq 1 \quad \mbox{and} \quad  m+p-2< 1.
\]

\noindent In fact, by using the assumptions above, we prove H\"older estimates with the optimal exponent
{\small{
\begin{equation}\label{Sharp_alpha}{\tag{Sharp}}
\alpha \defeq \min\left\{\alpha^{-}_0\max\left\{\frac{1}{\alpha^{-}_0(m+p-3)+1}, 1\right\}, \frac{(pq-n)r-pq}{q[(r-1)(m+p-2)+1]}\right\}.
\end{equation}}}
\noindent Now, we set the \textit{intrinsic scaling factor} $\theta$ as
\begin{equation}\label{Int-Scal}
    \theta \defeq  p-\alpha(m+p-3) = p -\alpha(m+p-2)\Big(1-\frac{1}{m+p-2}\Big).
\end{equation}

As in \cite{Ara20} and \cite{BJdaSR22}, we also assume an upper bound for the optimal exponent $\alpha$. This is the content of the next assumption.

\begin{Assumption}[Upper bound for $\alpha$]\label{assump_alpha} We assume that
\[
\alpha \leq \frac{p-1}{p+m-2}.
\]
\end{Assumption}

We notice that, if $m+p-2\geq 1$, then	
\[
1+\frac{p-1}{p+m-2}\le \theta \le p \quad \text{for} \quad \max\left\{1, \frac{2n}{n+2}\right\}< p < 2\quad \text{and} \quad m\geq1.
\]
On the other hand, if $m+p-2\leq1$, we have
\[
p\leq\theta\leq 1+\frac{p-1}{p+m-2} \quad \text{for} \quad \max\left\{1, \frac{2n}{n+2}\right\}<p<2 \quad \text{and} \quad m \approx 1.
\]

Under these conditions, our main result reads as follows:

\begin{theorem}\label{t1.2}
Let $u \in C_{\loc}(-1, 0; L^2_{\loc}(B_1))$ be a bounded weak solution of \eqref{1.1} in $Q^{-}_1$. Suppose that A\ref{assump_f}-A\ref{assump_alpha} hold true. Then $u \in C_{\loc}^{\alpha, \frac{\alpha}{\theta}}(Q^{-}_1)$, {\it i.e.}, there exists a universal constant $\mathrm{M}_0>0$ such that
\[
  \displaystyle  [u]_{C^{\alpha, \frac{\alpha}{\theta}}\left(Q^{-}_{\frac{1}{2}}\right)} \leq \mathrm{M}_0\left[\|u\|_{L^{\infty}(Q^{-}_1)} + \|f\|_{L^{q, r}(Q^{-}_1)}\right],
\]
where $\alpha \in (0, 1)$ is as in \eqref{Sharp_alpha}, $\theta$ is given by \eqref{Int-Scal}, and
\[
   [u]_{C^{\alpha, \frac{\alpha}{\theta}}\left(Q^{-}_{\frac{1}{2}}\right)} \defeq \sup_{0< \rho \leq \rho_0} \left(\inf_{(x_0, t_0) \in Q^{-}_{\rho_0}} \frac{\|u - u(x_0, t_0)\|_{L^{\infty}\left(Q^{-}_{\rho_0}(x_0, t_0) \cap Q^{-}_{\frac{1}{2}}\right)}}{\rho_0^{\alpha}}\right).
\]
\end{theorem}

\begin{remark}\label{r2.1}
In the proofs of Theorem x and Y, we require some smallness condition on $\|u\|_{L^{\infty}}$ and $\|f\|_{L^{q, r}}$, but those conditions are not restrictive. Indeed, for any $\sigma>0$ and $l>0$ fixed, we can find a positive constant $\rho =\rho(\sigma, l, \|u\|_{L^{\infty}}, \|f\|_{L^{q, r}})$ such that the function
\[
v_{\rho}(x,t)\defeq \rho^{l}u(\rho^{l}x,\rho^{\tau}t),
\]
satisfies the assumptions of Lemma \ref{l2.1}, where $\tau:=s(m-1)+2s(p-1)>0$,
\[
0<\rho \le \min\left\{ 1, \frac{1}{ \sqrt[l]{\|u\|_{L^{\infty}(Q^{-}_1)}}},\, \sqrt[\kappa_0]{\frac{\sigma}{\|f\|_{L^{q, r}(Q^{-}_1)}}} \right\},
\]
and
\[
\begin{array}{rcl}
  \kappa_0 & \defeq & l[2(p-1)+m]-\left(\frac{ln}{q}+\frac{(2p-1)l+l(m-1)}{r}\right) \\
   & = & l(m+p-2)\left[\left(1-\frac{1}{r}\right)+\frac{1}{r(m+p-2)}\right] + lp\left[1-\left(\frac{n}{pq}+\frac{1}{r}\right)\right]
\end{array}
\]
is a positive constant thanks to \eqref{w-cc}.
\end{remark}

We close this section commenting about the  optimal value of $\alpha_0$. Although it is still unknown, previous results indicate that $\alpha_0$ could be
\begin{equation}\label{SharpAlpha_Hom}
  \alpha_0 = \min\left\{1, \frac{p-1}{m+p-3}\right\}.
\end{equation}
In fact, the fundamental solution of
\[
  \frac{\partial u}{\partial t}-\Div(m\lvert u\rvert^{m-1}\lvert \nabla u \rvert^{p-2}\nabla u) = 0,
\]
is given by the Barenblatt function
\[
\mathcal{B}_{m, p}(x, t) \defeq \left\{
\begin{array}{cc}
 \frac{1}{t^{n\lambda_0}}\left[1-b(m, n, p)\left(\frac{|x|}{t^{\lambda_0}}\right)^{\frac{p}{p-1}}\right]_{+}^{\frac{p-1}{m+p-3}} & t>0 \\
 0 & t \le 0,
\end{array}
\right.
\]
where
\[
\lambda_0^{-1} = n(m+p-3) \quad \text{and} \quad b(m, n, p) = \frac{p-1}{p}\frac{m+p-3}{(m+p-2)}\lambda_0^{\frac{1}{p-1}},
\]
which leads to believe that $\alpha_0$ in \eqref{SharpAlpha_Hom} is optimal. We refer the reader to \cite[p. 2012]{Iva95}, \cite[Remark 3.4]{Iva00} and \cite[p. 299]{Stur18}.

Keeping that in mind, we can rewrite the second exponent in \eqref{Sharp_alpha} as
{\small{
\[
\frac{(pq-n)r-pq}{q[(r-1)(m+p-2)+1]} = \frac{p\left[1-\left(\frac{n}{pq}+\frac{1}{r}\right)\right]}{p\left[1-\left(\frac{n}{pq}+\frac{1}{r}\right)\right] + \left\{\left[\frac{3}{r}+m\left(1-\frac{1}{r}\right)+\frac{n}{q}\right]-2\right\}} \in (0, 1),
\]}}
which is well-defined according to \eqref{cc}.

\section{Auxiliary tools and results}

Next, we will introduce some auxiliary results. We start with the notion of weak solutions for our scenario.

\begin{definition}\label{d2.1} We say that a locally bounded function $u$ is a {\it weak solution} of \eqref{1.1} in $Q^{-}_1$, on condition that $u \in C_{\loc}(-1, 0; L_{\loc}^2(B_1))$ and $|u|^{\frac{m+p-2}{p-1}} \in L^{p}_{\loc}(-1, 0;W_{\text{loc}}^{1, p}(B_1))$, and for every compact set $\mathrm{K}\subset B_1$, every $[t_1, t_2] \subset (-1, 0]$ and $\psi\in H^1_{\loc}(-1, 0; L^2(\mathrm{K}))\cap L^{p}_{\loc}(-1,0; W_0^{1, p}(\mathrm{K}))$ there holds
\[
\displaystyle \left.\int_{\mathrm{K}} u \psi\,dx \right|^{t_2}_{t_1} + \int_{t_1}^{t_2} \int_{\mathrm{K}} \left[-u \frac{\partial \psi}{\partial t} + m|u|^{m-1}|\nabla u|^{p-2}\nabla u \cdot \nabla \psi \right]\,dx\,dt = \int_{t_1}^{t_2} \int_{\mathrm{K}} f \psi\,dx\,dt.
\]
Observe that the above integrals in this definition are convergent. Indeed, since
\[
\nabla |u|^{\frac{m+p-2}{p-1}} = \left(\frac{m+p-2}{p-1}\right)(\text{sgn}\,u)|u|^{\frac{m-p}{p-1}}u\nabla u \in L^{p}_{\loc}(B_1),
\]
which implies
\[
\big|\nabla |u|^{\frac{m+p-2}{p-1}}\big| = \left(\frac{m+p-2}{p-1}\right)|u|^{\frac{m-1}{p-1}}|\nabla u| \in L^{p}_{\loc}(B_1)
\]
hence,
\[
m|u|^{m-1}|\nabla u|^{p-2}\nabla u \in L^{p'}_{\loc}(B_1),
\]
where $p'$ is the H\"older conjugate of $p$.
\end{definition}

We highlight that under the condition
\[
\frac{\sigma+1}{\sigma+2}>\frac{1}{p}-\frac{1}{n} \quad \text{where} \quad \sigma=\frac{m-1}{p-1},
\]
weak solutions to the homogeneous problem are bounded. For more details, see the Ivanov's works \cite{Iva94-97} and \cite{Iva00}.

In the inhomogeneous scenario, O'Leary in \cite{Leary00} proved that a sufficient condition to local boundedness of weak solutions is the inclusion $L^{q, r}(\Omega_T) \subset L^{\beta, \beta}(\Omega_T)$ with $\beta > \frac{\mu}{\mu-1}$, where
\[
\mu = \left(\frac{n+2}{2n}\right)p+m-1,
\]
$\Omega_T$ is a parabolic domain of the form $\Omega\times (0,T]$, for $\Omega \subset \mathbb{R}^n$ an open set, and $T>0$. Hence, to guarantee the boundedness of weak solutions, it is enough to assume $r \ge q \ge \beta > \frac{\mu}{\mu-1}$ (cf. \cite{Stur18}). Moreover, we observe that if $n \ge 2$, then
\[
p> \max\left\{1, \frac{2n}{n+2}\right\} \quad \text{and} \quad m \ge 1 \quad \Rightarrow \quad \mu \ge1,
\]
and if $n=1$, we have
\[
p> 1 \quad \text{and} \quad m \ge \frac{3}{2} \quad \Rightarrow \quad \mu \ge1.
\]

As mentioned before, we argue by approximation methods, which require a compactness result for the weak solutions of \eqref{1.1}, that in our case is given by a Caccioppoli type estimate. To prove this result, we need to introduce an equivalent definition of weak solutions via the {\it Steklov average} of $u$, which is given by the function $u_h$ defined as follows:
\begin{equation}
	u_{h}\coloneqq \left\{
	\begin{array}{rcrcl}
	\displaystyle	\frac{1}{h}\int\limits_{t}^{t+h} u(.,\tau)d\tau, \ {if} \ t\in (0,T-h],\\
		0 ,\  {if} \ t \in (T-h,T],
	\end{array}
	\right.
\end{equation}
for $0<h<T$.

\begin{definition}\label{d2.2} A locally bounded function $u$ is called a {\it local weak solution} of \eqref{1.1} in $Q^{-}_1$, if $u \in C_{\loc}(-1, 0; L^2_{\loc}(B_1))$ and $|u|^{\frac{m+p-2}{p-1}} \in L^{p}_{\loc}(-1, 0;W_{\text{loc}}^{1, p}(B_1))$, and for every compact set $K\subset B_1$, every $[t_1, t_2] \subset (-1, 0]$ and $\psi\in H^1_{\loc}(-1, 0; L^2(K))\cap L^p_{\loc}(-1,0; W_{\loc}^{1, p}(K))$ there holds
\[
	\displaystyle \int_{K\times \{t\}} \left[(u_h)_t \psi + (m|u|^{m-1}|\nabla u|^{p-2}\nabla u)_h \cdot \nabla \psi \right]\,dx =\int_{K\times \{t\}} f_h \psi\,dx.
\]
\end{definition}

The Definition \ref{d2.2} can be obtained from the energy estimate given by \cite[Ch.3]{D93}. See also \cite[Proposition 2.1]{BJdaSR22} for more details. Now, we present the Caccioppoli type estimate.

\begin{proposition}[{\bf Caccioppoli estimate}]\label{Cacci_Est} Let $u$ be a weak solution to \eqref{1.1} in $Q^{-}_1$ and $\mathrm{K} \times [t_1, t_2] \subset B_1 \times (-1, 0]$. Then, there exists a constant $\mathrm{C} = \mathrm{C}(n, p, \mathrm{K}\times [t_1, t_2]) >0$  such that
{\small{
\[
\begin{array}{rcl}
 \displaystyle \sup\limits_{t \in (t_{1}, t_{2})} \int\limits_{\mathrm{K}} u^{2}\xi^p dx + \int\limits_{t_{1}}^{t_{2}}\int\limits_{\mathrm{K}} |u|^{m-1}|\nabla u|^{p}\xi^{p} dxdt  & \le &
\displaystyle \mathrm{C} \int\limits_{t_{1}}^{t_{2}}\int\limits_{\mathrm{K}} u^{2}\xi^{p-1}\left|\frac{\partial\xi}{\partial {t}}\right|dxdt+\int\limits_{t_{1}}^{t_{2}}\int\limits_{\mathrm{K}}|u|^{m-1}|u|^{p}|\nabla \xi|^{p} dxdt \vspace{0.2cm}\\
   & + & \displaystyle \mathrm{C}\|f\|_{L^{q,r}(Q^{-}_1)}^{2},
\end{array}
\]}}
for every test function $\xi \in C^{\infty}_{0}(\mathrm{K}\times (t_{1},t_{2}))$ in such a way that $0 \leq \xi \leq 1$.
\end{proposition}

\begin{proof}
The proof follows the same lines to the one presented in \cite[Proposition 3.1]{FS08}. In effect, the main idea consists of taking $\psi=u_{h}\xi^{p}$ as a test function in Definition \ref{d2.2} and, by letting $h$ goes to $0$, we invoke Young's inequality to conclude the proof. For instance, see  \cite[Proposition 2.1]{Ara20} and \cite[Proposition 22]{BJdaSR22}.
\end{proof}

Next, we present the interior H\"{o}lder regularity estimates for doubly nonlinear singular/degenerate PDEs as in \eqref{1.1}. Indeed, such estimates were addressed in \cite{Iva89}, \cite{Iva91}, \cite{Iva95}, \cite{Iva97}, \cite{Iva98}, \cite{Iva00}, \cite{PV93} and \cite{Vesp92}. Moreover, for the doubly nonlinear singular evolution models, that is, $1<p<2$ and $m\geq 1$, the regularity in H\"{o}lder spaces can be found in \cite[Theorem 2.4]{Ciani-Vespri00}, \cite[Theorem 1.1]{Iva1994}, \cite[Theorem 1.3]{Iva98}, \cite{Vesp22} and references therein.

\begin{theorem}\label{ThmHolderEst} Let $u$ be a locally bounded weak solution of \eqref{1.1}. Suppose that A\ref{assump_f} and A\ref{assump_w-cc} are in force. Then $u$ is locally H\"{o}lder continuous in $\Omega_T$, i.e., for every compact subset $\mathrm{K} \subset \Omega_T$, there exist constants $\gamma>1$ and $\alpha_0 \in (0, 1)$ such that
\[
|u(x_1, t_1)-u(x_2, t_2)| \leq \gamma\left(\frac{|x_1-x_2|^{\alpha_0}+ \|u\|_{L^{\infty}(\mathrm{K})}^{\frac{m+p-3}{p}}|t_1-t_2|^{\frac{\alpha_0}{\bar{p}}}}{(m,p)-\dist(\mathrm{K}, \partial_p \Omega_T)^{\alpha_0}} \right),
\]
for every $(x_1, t_1), (x_2, t_2) \in \mathrm{K}$, where $\bar{p}$ is a constant given by \eqref{p(m)}, $\Omega_T:= \Omega \times (0, T]$ and
\[
(m,p)-\dist(\mathrm{K}, \partial_p \Omega_T) \defeq \inf_{(x, t) \in \mathrm{K} \atop{(y, s) \in \partial_p \Omega_T}} \left\{|x-y|+\|u\|_{L^{\infty}(\mathrm{K})}^{\frac{m+p-3}{p}}|t-s|\right\}.
\]
\end{theorem}

Next, we recall the sharp regularity estimates for evolution $p$-Laplace equations of the form
\begin{equation}\label{Eqp-Lapla}
\frac{\partial u}{\partial t}-\Div({\mathcal A}(x, t)\lvert \nabla u \rvert^{p-2}\nabla u) = f(x, t) \quad \text{in} \quad Q^{-}_1,
\end{equation}
where $\max\left\{1, \dfrac{2n}{n+2}\right\} \leq p < \infty$, $f \in L^{q,r}(Q^{-}_1)$ satisfies the compatibility conditions \eqref{w-cc}, and the coefficient ${\mathcal A}$ is uniformly bounded, {\it i.e.},
\begin{equation}\label{ContCondp_laplace}
0 < \mathrm{L}_0 \le {\mathcal A}(x, t) \le \mathrm{L}_1< \infty \quad \text{in} \quad Q^{-}_1.
\end{equation}
This is the content of the next result.

\begin{theorem}\label{ThmSharp-p-Laplace} Assume that $u$ is a bounded weak solution to \eqref{Eqp-Lapla} in $Q^{-}_1$. Suppose further that assumptions A\ref{assump_f}, A\ref{assump_w-cc} and \eqref{ContCondp_laplace} there hold.. Then, $u$ is locally $C^{0, \hat{\alpha}}$ in space variable and $C^{0, \frac{\hat{\alpha}}{\hat{\theta}}}$ in time variable,  where
\[
\hat{\alpha} \defeq \frac{(pq-n)r-pq}{q[(p-1)r-(p-2)]} \qquad \text{and} \qquad \hat{\theta} \defeq 2\hat{\alpha} + (1-\hat{\alpha})p.
\]
Moreover, there exists a universal constant $\mathrm{C}_0>0$ such that
\[
  \displaystyle  [u]_{C^{\hat{\alpha}, \frac{\hat{\alpha}}{\hat{\theta}}}\left(Q^{-}_{\frac{1}{2}}\right)} \leq \mathrm{C}_0\left[\|u\|_{L^{\infty}(Q^{-}_1)} + \|f\|_{L^{q, r}(Q^{-}_1)}\right].
\]
\end{theorem}

The proof of Theorem \ref{ThmSharp-p-Laplace} for equations with constant coefficients, and $p \geq 2$ can be found in \cite[Theorem 3.4]{TU14}. For the case of continuous and uniformly bounded variable coefficients as in \eqref{Eqp-Lapla}-\eqref{ContCondp_laplace}, one can obtain similar results, see for instance  \cite[Section 4]{TU14}. Although the authors in \cite{TU14} only deal with the degenerate case $p \geq 2$, their arguments also work in the case $\max\left\{1, \frac{2n}{n+2}\right\} \leq p < 2$. Indeed, the only modification we need to do is in the control of the time derivative in \cite[Lemma 3.1]{TU14}. We can use the estimates along the lines of \cite[Theorem 1.5]{DPZZ20} (see also \cite[Corollary 2.5]{FPM22} and \cite[Chapter 7]{AMS04}), which works in our range of $p$, instead of the estimates in \cite{Lind08}, that are proved for the degenerate case.

The next result plays a crucial role in our arguments. Under a smallness regime on the $L^{q,r}$-norm of the source term, we show that weak solutions to \eqref{1.1} are close to the solution of the homogeneous $(m,p)$-Laplacian, in a suitable sense. This fact will allow us to import some regularity properties back to our model.

\begin{lemma}[{\bf $(m, p)$-Approximation}]\label{l2.1}
Let $u \in C_{\loc}(-1, 0; L^2_{\loc}(B_1))$ be a weak solution to \eqref{1.1} in $Q^{-}_1$ with
$\|u\|_{L^{\infty}(Q^{-}_1)} \leq 1$ and $f \in L^{q, r}(Q_1^{-})$. Given $\delta >0$, there exists $\varepsilon= \varepsilon(p,n,m,\delta)>0$ such that, if
\[
\displaystyle \|f\|_{L^{q, r}(Q^{-}_1)}\leq \varepsilon,
\]
then, we can find a function $h \in C_{\loc}^{\alpha_0, \frac{\alpha_0}{\bar{p}}}(Q^{-}_1)$ such that
\begin{equation}\label{3.1}
\sup_{Q^{-}_{\frac{1}{2}}}|u-h| <\delta.
\end{equation}
\end{lemma}

\begin{proof} We argue by contradiction. Suppose the thesis of the lemma fails, {\it i.e.}, there exist $\delta_0$ and sequences
\[
(u_{j})_{j} \in C_{\loc}(-1, 0;L^2_{\loc}(B_1)) \quad \text{with} \quad  |u_{j}|^{\frac{m+p-2}{p-1}} \in L^{p}_{\loc}(-1, 0;W_{\loc}^{1,p}(B_{1}))
\]
and $(f_{j})_{j} \in L^{q,r}(Q^{-}_{1})$ such that for all $j$
\begin{equation}\label{Eq2.9}
    \frac{\partial u_{j}}{\partial t}-\Div(m|u_j|^{m -1}|\nabla u_j|^{p-2} \nabla u_j) = f_{j}(x, t) \quad \text{in} \quad Q^{-}_{1}
\end{equation}
 with
\begin{equation}\label{Eq2.10}
\lvert\lvert u_{j}\rvert\rvert_{L^{\infty}(Q^{-}_{1})} \leq 1 \qquad \text{and} \qquad \lvert \lvert f_{j} \rvert\rvert_{L^{q,r}(Q^{-}_{1})} \leq\frac{1}{j},
\end{equation}
but
\begin{equation}\label{Eq2.11}
\sup_{Q^{-}_{\frac{1}{2}} }| u_{j}-h | \geq \delta_{0} \; \;  \text{for all} \; j,
\end{equation}
for every $h \in C_{\text{loc}}^{\alpha_0, \frac{\alpha_0}{\bar{p}}}(Q^{-}_1)$. Since $u_{j}$ satisfies the Cacciopolli estimate (Proposition \ref{Cacci_Est}), that is,
{\footnotesize{
\[
\begin{array}{rcl}
 \displaystyle \sup\limits_{-1<t < 0} \int\limits_{B_1} {u_j}^{2}\xi^p dx + \int\limits_{-1}^{0}\int\limits_{B_1} |u_j|^{m-1}|\nabla u_j|^{p}\xi^{p} dxdt
 & \le & \displaystyle \mathrm{C} \int\limits_{-1}^{0}\int\limits_{B_1} {u_j}^{2}\xi^{p-1}\left|\frac{\partial\xi}{\partial {t}}\right|dxdt+\int\limits_{-1}^{0}\int\limits_{B_1}|u_j|^{m+p-1}|\nabla \xi|^{p} dxdt \vspace{0.15cm}\\
 & + & \displaystyle \mathrm{C}\|f_j\|_{L^{q,r}(Q^{-}_1)}^{2},
\end{array}
\]}}
where $\xi \in C^{\infty}_{0}(Q^{-}_{1})$ is a cutoff function such that $0\le \xi \le 1$, with $\xi = 1$ in $Q^{-}_{\frac{1}{2}}$ and $\xi = 0$ in $\partial_{p} Q^{-}_{1}$. From \eqref{Eq2.10} and the inequality above, we obtain
\begin{equation*}
    \sup\limits_{-1<t<0} \int_{B_{1}} u_j^{2}\xi^p dx + \int_{-1}^{0}\int_{B_{1}} |u_j|^{m-1}|\nabla u_j|^{p}\xi^{p} dxdt \leq \bar{\mathrm{C}}.
\end{equation*}

Now, for $j \in \mathbb{N}$, let us introduce the sequence $v_{j}= {|u_j|}^{\frac{m+p-2}{p-1}}$. Direct calculation yields
\[
|\nabla v_{j}|^{p}=\left(\frac{m+p-2}{p-1}\right)^p|u_j|^{\frac{p(m-1)}{p-1}}|\nabla u_j|^p
\]
This implies,
\begin{align*}
\|\nabla v_{j} \|_{L^p \left(Q^{-}_{\frac{1}{2}}\right)}^{p} & = \int_{-\frac{1}{2^{\theta}}}^{0}\int_{B_{1}}\lvert \nabla v_{j} \rvert^{p}dxdt \leq \int_{-1}^{0}\int_{B_{1}}|u_{j}|^{\frac{p(m-1)}{p-1}}|\nabla u_{j}|^{p}\xi^{p}dxdt \\
& \leq \int_{-1}^{0}\int_{B_{1}}|u_{j}|^{(m-1)}|\nabla u_{j}|^{p}\xi^{p}dxdt \leq \bar{\mathrm{C}}
\end{align*}

Hence, up to a subsequence, we conclude
\begin{equation}\label{Eq2.12}
\nabla v_{j} \rightharpoonup \varphi \quad \text{weakly in} \; \; L^{p}\left(Q^{-}_{\frac{1}{2}}\right).
\end{equation}

Furthermore, it follows from the sequence $(u_{j})_j$ is equicontinuous, then by employing the Arzel\`{a}-Ascoli Theorem, up to a subsequence, we get that $u_{j}$ converges uniformly to $  u_{\infty}$ in $Q^{-}_{\frac{1}{2}}.$ In particular, we have the point-wise convergence, that is,
\begin{equation}\label{Eq2.13}
v_{j}={|u_j|}^{\frac{m+p-2}{p-1}} \to {|u_{\infty}|}^{\frac{m+p-2}{p-1}} \defeq v,
\end{equation}
and we can also identity $\varphi= \nabla v$ by using \eqref{Eq2.12} and \eqref{Eq2.13}.

Finally, by passing the limit as $j \to \infty$ in \eqref{Eq2.9}, we obtain that $u_\infty$ satisfies,
\begin{equation}\label{EqHomProb}
      \frac{\partial (u_\infty)}{\partial t}-\Div(m|u_\infty|^{m -1}|\nabla u_\infty|^{p-2} \nabla u_\infty) = 0 \quad \text{in} \quad Q^{-}_{\frac{1}{2}}.
\end{equation}
It follows from \cite{Iva95} and \cite{PV93} that $u_\infty \in C_{\text{loc}}^{\alpha_0, \frac{\alpha_0}{\bar{p}}}(Q^{-}_1)$  (in the parabolic sense), for $0<\alpha_0\leq1$, and $\bar{p}>0$ as in \eqref{p(m)}. Moreover, because of the uniform convergence, we have for $j$ sufficiently large,
\[
\sup_{Q^{-}_{\frac{1}{2}}}|u_j-u_\infty| \leq \delta_0,
\]
which yields to a contradiction with \eqref{Eq2.11}.

\end{proof}

\section{Geometric $\alpha-$H\"{o}lder estimates}

In this section, we use the $(m, p)$-Approximation Lemma \ref{l2.1} to put forward a geometric iteration that will result in the desired H\"older estimates.

\begin{proposition}\label{1stStepInduc} Let $u \in C_{\loc}(-1, 0; L^2_{\loc}(B_1))$ be a weak solution to \eqref{1.1} in $Q^{-}_{1}$, with $\| u\|_{L^{\infty}\left(Q^{-}_{1}\right)} \leq 1$. There exist $\varepsilon > 0$ and $\lambda \in \left(0, \frac{1}{4}\right]$ both depending only on universal parameters $m,n,p$ and $\alpha$, such that if
\[
\|f\|_{L^{q,r}(Q^{-}_{1})} < \varepsilon,
\]
then
\[
	\|u\|_{L^{\infty}\left(Q^{-}_{\lambda}\right)}\leq \lambda^{\alpha} \qquad  \mbox{provided} \qquad |u(0,0)|\leq \frac{\lambda^{\alpha}}{4}.
\]
\end{proposition}

\begin{proof} Fix $\delta \in (0, 1)$ to be determined later. By applying the $(m, p)$-Approximation Lemma \ref{l2.1}, we can find $\varepsilon>0$ and a function $h \in C_{\text{loc}}^{\alpha_0, \frac{\alpha_0}{\bar{p}}}(Q^{-}_1)$ such that
\[
\|u-h\|_{L^{\infty}\left(Q^{-}_{\frac{1}{2}}\right)}\leq\delta.
\]
From the regularity available for $h$, we have in particular that
{\small{
	\begin{equation}\label{Eq_phi}
			  			|h(x,t)-h(y,s)|\leq \gamma\left( |x-y|^{\alpha_0} + \|h\|_{L^{\infty}\left(Q^{-}_{\frac{1}{2}}\right)}^{\frac{m+p-3}{p}}\sqrt[\bar{p}]{|t-s|^{\alpha_0}}\right) \quad \forall\,\,(x,t),(y,s) \in Q^{-}_{\frac{1}{2}}.
	\end{equation}
}}
Notice that by choosing $\lambda \in \left(0,  \frac{1}{4}\right]$, we get that $Q^{-}_{\lambda}\subset Q^{-}_{\frac{1}{2}}$. For the sake of clarity, we split the proof into two cases.

\bigskip	
\noindent\textit{Case 1:} First, we analyze the case $p+m-2\geq 1$. It follows from $A$\ref{assump_alpha} that $\alpha\leq \frac{p-1}{p+m-2}$, which implies
\[
\frac{2(p-1)}{p+m-2}\leq 1+\frac{p-1}{p+m-2}\leq\theta\leq p.
\]
Now, we state that, if  $m>1$ and $\max\left\{1, \frac{2n}{n+2}\right\}<p\leq 2$, then $h$ satisfies
\[
\sup\limits_{(x,t)\in Q^{-}_{\lambda}}|h(x,t)-h(0,0)|\leq \mathrm{C}\lambda^{\frac{\alpha_0(p-1)}{p+m-2}}.
\]
Indeed, consider $(x,t) \in Q^{-}_{\lambda}$, and notice that $\alpha_0\geq \frac{\alpha_0(p-1)}{p+m-2}$. Since $\max\left\{1, \frac{2n}{n+2}\right\}<p\leq2$, we obtain $\frac{2\alpha_0(p-1)}{p(p+m-2)} \geq \frac{\alpha_0(p-1)}{p+m-2}$. By using the fact that $\theta \geq \frac{2(p-1)}{p+m-2}$ combined with \eqref{Eq_phi}, we can estimate
\begin{equation}\label{Second_Est}
	\begin{array}{rcl}
		| h(x,t) -h(0,0)| & \leq & | h(x,t)-h(0,t)| + | h(0,t)-h(0,0)| \\
		& \leq & c_{1}|x - 0|^{\alpha_0} +c_{2}| t-0|^{\frac{\theta\alpha_0}{\bar{p}}} \\
		& \leq & c_{1}\lambda^{\alpha_0} +c_{2}\lambda^{\frac{\theta\alpha_0}{p}} \\
		& \leq & c_{1}\lambda^{\frac{\alpha_0(p-1)}{p+m-2}} +c_{2}\lambda^{\frac{2\alpha_0(p-1)}{p(p+m-2)}} \\
		& \leq & \max\{c_1, c_2\}\lambda^{\frac{\alpha_0(p-1)}{p+m-2}}.
	\end{array}
\end{equation}
Therefore, we use the $(m, p)$-Approximation Lemma \ref{l2.1}  with the estimate above to obtain
\begin{equation}\label{EQ_Est-alpha-Hold-2}
\begin{array}{rcl}
				\sup\limits_{Q^{-}_{\lambda}}\lvert u \rvert & \le & \sup\limits_{Q^{-}_{\frac{1}{2}}} \vert u-h \rvert + \sup\limits_{Q^{-}_{\lambda}}\lvert h(x,t)-h(0,0)\rvert + \lvert u(0,0)-h(0,0)\rvert + \lvert u(0,0) \rvert\\
				& \le & 2\delta + \mathrm{C}\lambda^{\frac{\alpha_0(p-1)}{p+m-2}}+\frac{\lambda^{\alpha}}{4}.
\end{array}
\end{equation}
Finally, we take
\[
\lambda \in \left(0, \,\min\left\{\frac{1}{4}, \left( \frac{1}{4\mathrm{C}}\right)^{\frac{p+m-2}{\alpha_0(p-1)-\alpha(p+m-2)}}\right\}\right] \quad \text{and} \quad \delta \in \left(0, \frac{\lambda^{\alpha}}{4}\right] \quad  \text{if} \quad \max\left\{1, \frac{2n}{n+2}\right\}<p\leq 2,
\]
and we plug it into \eqref{EQ_Est-alpha-Hold-2} to get the desired estimate.

\bigskip
	
\noindent\textit{Case 2:} Now, we consider the case $p+m-2\leq 1$. First, we use $A$\ref{assump_alpha} combined with \eqref{Intr_Scalling} to get
\[
p\leq\theta\leq 1+\frac{p-1}{p+m-2}.
\]
Here, we suppose $m>1$ and use the fact that $p \leq \theta$ to estimate
\[
\begin{array}{rcl}
	|h(x,t) -h(0,0)| & \le & |h(x,t)-h(0,t)| + | h(0,t)-h(0,0)| \\
	& \leq & c_{1}|x - 0|^{\alpha_0} +c_{2}|t-0 |^{\frac{\alpha_0\theta}{\bar{p}}} \\
	& \leq & c_{1}\lambda^{\alpha_0}+c_{2}\lambda^{\frac{\alpha_0\theta}{p}} \\
	& \leq & \max\{c_1, c_2\}\lambda^{\alpha_0}.
\end{array}
\]
For the case $m>1$, we obtain the following estimate by applying the $(m, p)$-Approximation Lemma \ref{l2.1}.
\begin{equation}\label{EQ_Est-alpha-Hold-8}
			\begin{array}{rcl}
				\sup\limits_{Q^{-}_{\lambda}}|u|& \le & \sup\limits_{Q^{-}_{\frac{1}{2}}} |u-h| + \sup\limits_{Q^{-}_{\lambda}}|h(x,t)-h(0,0)| + |u(0,0)-h(0,0)| + | u(0,0) |\\
				& \le & 2\delta + \mathrm{C}\lambda^{\alpha_0}+\frac{\lambda^{\alpha}}{4}.
			\end{array}
\end{equation}
Next, we can take in \eqref{EQ_Est-alpha-Hold-8}
\[
\lambda \in \left(0, \,\min\left\{\frac{1}{4}, \left( \frac{1}{4\mathrm{C}}\right)^{\frac{1}{\alpha_0-\alpha}}\right\}\right] \quad \text{and} \quad \delta \in \left(0, \frac{\lambda^{\alpha}}{4}\right] \quad \text{for}  \quad m>1
\]
to obtain the desired estimate.

On the other hand, if $m=1$, we have $\theta=p(1-\alpha)+2\alpha\leq2$. Hence,
\[
\begin{array}{rcl}
	|h(x,t) -h(0,0)| & \le & |h(x,t)-h(0,t)| + | h(0,t)-h(0,0)| \\
	& \leq & c_{1}|x - 0| +c_{2}|t-0|^{\frac{\theta}{\bar{p}}} \\
	& \leq & c_{1}\lambda+c_{2}\lambda^{\frac{\theta}{2}} \\
	& \leq & \max\{c_1, c_2\}\lambda^{\frac{\theta}{2}}.
\end{array}
\]

Now, for $m=1$ we have
\begin{equation}\label{EQ_Est-alpha-Hold-10}
\begin{array}{rcl}
\sup\limits_{Q^{-}_{\lambda}}\lvert u \rvert & \le & \sup\limits_{Q^{-}_{\frac{1}{2}}} \vert u-h \rvert + \sup\limits_{Q^{-}_{\lambda}}\lvert h(x,t)-h(0,0)\rvert + \lvert u(0,0)-h(0,0)\rvert + \lvert u(0,0) \rvert\\
& \le & 2\delta + \mathrm{C}\lambda^{\frac{\theta}{2}}+\frac{\lambda^{\alpha}}{4}.
\end{array}
\end{equation}

Finally, if $m=1$, we take in \eqref{EQ_Est-alpha-Hold-10}
\[
\lambda \in \left(0, \,\min\left\{\frac{1}{4}, \left( \frac{1}{4\mathrm{C}}\right)^{\frac{2}{\theta-2\alpha}}\right\}\right] \quad \text{and} \quad \delta \in \left(0, \frac{\lambda^{\alpha}}{4}\right].
\]
This finishes the proof.
\end{proof}

In the sequel, we iterate the previous result in parabolic $\lambda$-adic cylinders.

\begin{proposition}\label{induction} Let $u \in C_{\loc}(-1, 0; L^2_{\loc}(B_1))$ be a weak solution to \eqref{1.1} in $Q^{-}_{1}$, with $\| u\|_{L^{\infty}\left(Q^{-}_{1}\right)} \leq 1$.
Suppose further that
\[
\|f\|_{L^{q,r}(Q^{-}_{1})} < \varepsilon,
\]
where $\varepsilon$ comes from the Proposition \ref{1stStepInduc}. Then,
\begin{equation}\label{Induc-k-Step}
  |u(0,0)|\leq \frac{\lambda^{\alpha k}}{4} \qquad \mbox{for each} \quad k \in \mathbb{N} \quad \mbox{implies} \quad \|u\|_{L^{\infty}\left(Q^{-}_{\lambda^{k}}\right)}\leq \lambda^{\alpha k}.
\end{equation}
\end{proposition}

\begin{proof} We resort to an induction argument. It follows from Proposition \ref{1stStepInduc} that the case $k=1$ is true. Now, we suppose that the statement have been proved for $j = 1, 2, \ldots, k$. Next, we address the case $j = k+1$. Suppose $|u(0,0)| \leq \lambda^{(k+1)\alpha}/4$ holds. We define the auxiliary function $v_k: Q^{-}_{1}\to \mathbb{R}$ by
\[
v_k(x,t) \defeq \frac{u(\lambda^{k}x,\lambda^{k\theta}t)}{\lambda^{k\alpha}}.
\]
Notice that $v_k$ is a weak solution of
\[
\frac{\partial v_k}{\partial t}-\Div(\mathcal{A}_k(v_k,\nabla v_k))=f_k(x,t) \quad \text{in}  \quad  Q^{-}_{1},
\]
where
\[
\left\{
\begin{array}{rcl}
  \mathcal{A}_k(s, \xi) & \defeq  & \lambda^{-k\alpha(m+p-2)+k(p-1)}\mathcal{A}(\lambda^{k\alpha}s,\lambda^{(\alpha-1)k}\xi) \\
  \mathcal{A}(s, \xi) & \defeq  & m|s|^{m-1}|\xi|^{p-2}\xi\\
  f_k(x, t) & \defeq & \lambda^{-k\alpha(m+p-2)+k(p-1)+k}f(\lambda^{k}x,\lambda^{k\theta}t).
\end{array}
\right.
\]
Hence, by using the induction hypothesis, we can guarantee that
\[
\|v_k\|_{L^{\infty}\left(Q^{-}_{1}\right)} \leq 1 \quad \mbox{and} \quad |v_k(0,0)|=\frac{|u(0,0)|}{\lambda^{k\alpha}}\leq \frac{\lambda^{(k+1)\alpha}}{4\lambda^{k\alpha}}=\frac{\lambda^{\alpha}}{4}.
\]

Moreover, $f_k$ satisfies
\[
\begin{array}{rcl}
  \| f_k \|_{L^{q,r}(Q_{1}^{-})}^{r} & = & \displaystyle  \int\limits_{-1}^{0}\Bigg(\int\limits_{B_{1}}\lvert f_k(x,t)\rvert^{q} dx\Bigg)^{\frac{r}{q}}dt\\
   & = & \displaystyle \int\limits_{-1}^{0}\Bigg(\int\limits_{B_{1}}\lambda^{\left(-k\alpha(m+p-2)+k(p-1)+k\right)q}\lvert f(\lambda^{k}x,\lambda^{k\theta} t)\rvert^{q}dx\Bigg)^{\frac{r}{q}}dt\\
   & = & \displaystyle \int\limits_{-1}^{0}\Bigg(\int\limits_{B_{\lambda^{k}}}\lambda^{\left[\left(-k\alpha(m+p-2)+k(p-1)+k\right)q-nk\right]}\lvert f(z,\lambda^{k\theta} t)\rvert^{q}dz\Bigg)^{\frac{r}{q}} dt\\
   & = & \displaystyle \lambda^{\left[\left(-k\alpha(m+p-2)+k(p-1)+k\right)q-nk\right]\frac{r}{q}} \lambda^{-k\theta}\int\limits_{-\lambda^{k\theta}}^{0}\Bigg(\int\limits_{B_{\lambda^{k}}}|f(z,\tau)|^{q} dz\Bigg)^{\frac{r}{q}}d\tau.
\end{array}
\]
Thus,
\[
\|f_k\|_{L^{q,r}(Q^{-}_{1})}^{r} =  \lambda^{\left[\left(-k\alpha(m+p-2)+k(p-1)+k\right)q-nk\right]\frac{r}{q}-k\theta}\|f\|_{L^{q,r}(Q^{-}_{\lambda^{k}})}^{r}.
\]
Furthermore, we obtain that
{\small{
\[
\left[\left(-k\alpha(m+p-2)+k(p-1)+k\right)q-nk\right]\frac{r}{q}-k\theta \geq 0 \iff \alpha \leq \frac{r(pq-n)-pq}{q[(m+p-2)r-(m+p-3)]}.
\]}}

Now, since $\lambda \in \left(0, \frac{1}{4}\right]$, we can conclude
\[
\|f_k\|_{L^{q,r}(Q^{-}_{1})}\leq \|f\|_{L^{q,r}(Q^{-}_{\lambda^{k}})} \leq \|f\|_{L^{q,r}(Q^{-}_{1})}\leq \varepsilon,
\]
which implies that $v_k$ satisfies the hypothesis of Proposition \ref{1stStepInduc}. Therefore
\[
 \|v_k\|_{L^{\infty}\left(Q^{-}_{\lambda}\right)} \leq \lambda^{\alpha},
\]
by rescaling back to the unitary setting, we get
\[
\|u\|_{L^{\infty}\left(Q^{-}_{\lambda^{(k+1)}}\right)}\leq \lambda^{\alpha (k+1)},
\]
and complete the proof.
\end{proof}

\section{Proof of the Theorem \ref{t1.2}}

In the next result, we control the oscillation of $u$ in the intrinsic cylinders with continuous radii.

\begin{proposition}\label{rho} Let $u \in C_{\loc}(-1, 0; L^2_{\loc}(B_1))$ be a bounded weak solution of \eqref{1.1} in $Q^{-}_{1}$ and $\lambda>0$ as in Proposition \ref{1stStepInduc}. Then, for $\rho \in (0, \lambda)$ and $\mathrm{C}>0$ a universal constant, we have
\[
|u(0,0)|\leq \frac{\rho^{\alpha}}{4} \qquad \text{implies} \qquad \lvert \lvert u \rvert \rvert_{L^{\infty}\left(Q^{-}_{\rho}\right)}\leq \mathrm{C}\rho^{\alpha}.
\]
\end{proposition}

\begin{proof}

Recall that from Remark \ref{r2.1}, we can suppose
\[
\|u\|_{L^{\infty}(Q^{-}_{1})} \leq 1 \quad  \mbox{and} \quad \|f\|_{L^{q,r}(Q^{-}_1)} \leq\varepsilon.
\]
Now, given $\rho \in (0, \lambda)$, and let $k \in \mathbb{N}$ be such that
\[
\lambda^{k+1}<\rho \leq \lambda^{k},
\]
which implies that
\[
|u(0,0)|\leq \frac{\rho^{\alpha}}{4}\leq \frac{\lambda^{k\alpha}}{4}.
\]
For this reason, we can apply the Proposition \ref{induction}  to obtain that
\[
\|u\|_{L^{\infty}\left(Q^{-}_{\lambda^{k}}\right)} \leq \lambda^{k\alpha}.
\]

Therefore, we can conclude that
\[
\|u\|_{L^{\infty}(Q^{-}_{\rho})}\leq \|u\|_{L^{\infty}\left(Q^{-}_{\lambda^{k}}\right)} \leq \lambda^{k\alpha} \leq \Big(\frac{\rho}{\lambda} \Big)^{\alpha} = \mathrm{C}\rho^{\alpha}.
\]
\end{proof}

In what follows, we provide the proof of the Theorem \ref{t1.2}.

\begin{proof}[{\it Proof of Theorem \ref{t1.2}}]

First, let us recall that to proof the statement of the Theorem  \ref{t1.2}, it is enough to find a constant $\mathrm{M}_0>0$ such that
\begin{equation}\label{MainEstThm}
  \|u-u(0,0)\|_{L^{\infty}\left(Q^{-}_{\rho}\right)} \leq \mathrm{M}_0\rho^{\alpha}.
\end{equation}
Define the quantity
\begin{equation}\label{Def-mu}
\mu \defeq (4|u(0,0)|)^{1/\alpha}\geq 0,
\end{equation}
and consider $\rho \in (0, \lambda)$. We split the proof into three cases:\\
\noindent{\it Case 1:} $\rho \in [\mu, \lambda)$: Notice that in this case
\[
   |u(0,0)|= \frac{\mu^{\alpha}}{4} \leq \frac{\rho^{\alpha}}{4}.
\]
Hence, we can apply Proposition \ref{rho} to obtain
\begin{equation}\label{FirstEstThm}
  \displaystyle  \sup\limits_{Q^{-}_{\rho}} |u(x,t)-u(0,0)|\leq \mathrm{C}\rho^{\alpha} + |u(0,0)| \leq \Big(\mathrm{C}+\frac{1}{4} \Big)\rho^{\alpha}.
\end{equation}

\bigskip

\noindent{\it Case 2:} $\rho \in (0, \mu)$: We introduce the auxiliary function $w: Q^{-}_{1} \to \mathbb{R}$ given by
\[
w(x,t) \defeq \frac{u(\mu x, \mu^{\theta}t)}{\mu^{\alpha}}.
\]
Observe that, we have $|w(0, 0)| = \frac{1}{4}$. Moreover, $w$ solves
\[
\frac{\partial w}{\partial t}-\Div(\mathcal{A}_{\mu}(w,\nabla w))= f_{\mu}(x,t) \qquad \text{in} \qquad Q^{-}_{1},
\]
where
\[
\left\{
\begin{array}{rcl}
  \mathcal{A}_{\mu}(s, \varsigma) & \defeq & \mu^{-[\alpha(m-1)+(\alpha-1)(p-1)]}\mathcal{A}(\mu^{\alpha}s, \mu^{\alpha -1}\varsigma) \\
  \mathcal{A}(s, \varsigma) & \defeq & m|s|^{m-1}|\varsigma|^{p-2}\varsigma\\
  f_{\mu}(x, t) & \defeq & \mu^{-[\alpha(m-1)+(\alpha-1)(p-1)]+1}f(\mu x, \mu^{\theta} t).
\end{array}
\right.
\]

Applying Proposition \ref{rho} again, we obtain (recall that $|u(0,0)|=\frac{\mu^{\alpha}}{4}$)
\begin{equation}
\|w\|_{L^{\infty}(Q^{-}_{1})} = \frac{1}{\mu^{\alpha}}\|u\|_{L^{\infty}(Q^{-}_{\mu})}\leq \frac{\mathrm{C}\mu^{\alpha}}{\mu^{\alpha}}= \mathrm{C},
\end{equation}
Thanks to the uniform estimate above, and Theorem \ref{ThmHolderEst}, we can guarantee  the existence of a radius $\rho_{0}>0$, in which
\[
|w(x,t)| \geq \frac{1}{8} \quad \text{for all} \quad (x,t) \in \overline{Q^{-}_{\rho_{0}}}\subset Q^{-}_{\frac{1}{4}} .
\]
In fact, we use the H\"{o}lder estimates provided by  Theorem \ref{ThmHolderEst} to get
\begin{equation}
\begin{array}{rcl}\label{estimate0_Prop5.1}
  \frac{1}{4} & = & |w(0, 0)| \\
   & \le & |w(x_{\mathrm{Min}}, t_{\mathrm{Min}})-w(0, 0)|+|w(x_{\mathrm{Min}}, t_{\mathrm{Min}})| \\
   & \le & \gamma  \left(\frac{|x_{\mathrm{Min}}|^{\alpha_0}+ |t_{\mathrm{Min}}|^{\frac{\alpha_0}{\bar{p}}}}{(m,p)-\dist(\overline{Q^{-}_{\rho_{0}}}, \partial_p Q_1^{-})^{\alpha_0}} \right) +|w(x_{\mathrm{Min}}, t_{\mathrm{Min}})|\\
   & \le & \gamma  \left(\frac{4}{3}\right)^{\alpha_0}\left(\rho_0^{\alpha_0}+ \rho_0^{\frac{\alpha_0\theta}{\bar{p}}} \right) + \frac{1}{8}
\end{array}
\end{equation}
where $(x_{\mathrm{Min}}, t_{\mathrm{Min}}) \in \overline{Q^{-}_{\rho_{0}}}$ is a minimum point of $w$ and $(m,p)-\dist \left(\overline{Q^{-}_{\rho_{0}}}, \partial_p Q_1^{-}\right) \geq \frac{3}{4}$.

To finish this case, we need to analyze the scenarios $m+p-3 \geq 0$ and $m+p-3 < 0$. For the sake of clarity, we divide the proof into two subcases.
\bigskip

{\it Subcase 1.} Let us assume $m+p-3 \geq 0$, we have $\theta \leq p \leq 2$, which implies $\frac{\theta}{\bar{p}} \leq \frac{p}{\bar{p}} \leq 1$. Hence, we can conclude that
\begin{equation}\label{estimate_Prop5.1}
\frac{\alpha_0 \theta}{\bar{p}} \leq \alpha_0.
\end{equation}

We use \eqref{estimate_Prop5.1} combined with  \eqref{estimate0_Prop5.1} to obtain
\[
\frac{1}{4} \leq \gamma \left(2\left(\frac{4}{3}\right)^{\alpha_0} {\rho_0}^{\frac{\alpha_0 \theta}{\bar{p}}}\right) + \frac{1}{8}.
\]

Therefore,
\begin{equation}\label{Eq_rho_1}
  \rho_0 \geq \left( \frac{1}{16 \gamma} \left(\frac{3}{4}\right)^{\alpha_0}\right)^{\frac{\bar{p}}{\alpha_0 \theta}}.
\end{equation}

\bigskip
{\it Subcase 2.} On the other hand, suppose $m+p-3 < 0$ holds. Here, we also examine the cases $m>1$ and $m=1$. First, let us consider $m>1$, then
\[
p \leq \theta \leq 1+ \frac{p-1}{p+m-2} < 2.
\]

Thus, $\frac{p}{\bar{p}} \leq \frac{\theta}{p} \leq \frac{2}{p}$. Hence, it follows from the definition \eqref{p(m)} that $\frac{\theta}{\bar{p}} \geq 1$, {\it i.e.,}
\begin{equation} \label{estimate1_Prop5.1}
\frac{\alpha_0 \theta}{\bar{p}} \geq \alpha_0.
\end{equation}

We use the estimate \eqref{estimate1_Prop5.1} in \eqref{estimate0_Prop5.1} to get that
\[
\frac{1}{4} \leq \gamma \left(2\left(\frac{4}{3}\right)^{\alpha_0} {\rho_0}^{\alpha_0 }\right) + \frac{1}{8}.
\]

Thus,
\begin{equation}\label{Eq_rho_2}
\rho_0 \geq \left( \frac{1}{16 \gamma} \left(\frac{3}{4}\right)^{\alpha_0}\right)^{\frac{1}{\alpha_0}}.
\end{equation}

Now, we suppose that $m=1$, then $\bar{p} =2$. Hence, $p \leq \theta \leq 2$ which implies, $\frac{p}{\bar{p}} \leq \frac{\theta}{\bar{p}} \leq 1$. Hence $\frac{\alpha_0 \theta}{\bar{p}} \leq \alpha_0$. 	Similarly to \eqref{estimate_Prop5.1}, we obtain that
\begin{equation}\label{Eq_rho_3}
\rho_0 \geq \left( \frac{1}{16 \gamma} \left(\frac{3}{4}\right)^{\alpha_0}\right)^{\frac{2}{\alpha_0 \theta}}.
\end{equation}

Therefore, $w$ satisfies the
$$
\frac{\partial w}{\partial t}-\Div(\mathfrak{A}_0(x,t, w)|\nabla w|^{p-2}\nabla w))= f(x, t) \quad \text{in} \quad  Q^{-}_{\rho_{0}},
$$
where $f \in L^{q, r}(Q^{-}_1)$ and $(x, t) \mapsto \mathfrak{A}_0(x,t, w)$ is a continuous function satisfying
\[
0<m\left(\frac{1}{8}\right)^{m-1}\le \mathfrak{A}_0(x,t, w) \leq m \lvert \lvert w \rvert \rvert_{L^{\infty}(Q^{-}_{1})}^{m-1}\le m\mathrm{C}^{m-1}< \infty.
\]
Hence, we can think of the equation above as an evolution $p$-Laplacian problem as in \cite{TU14} and \cite{TU21}.

Therefore, it follows from Theorem \ref{ThmSharp-p-Laplace} that $w \in C^{\hat{\alpha}, \frac{\hat{\alpha}}{\hat{\theta}}}(Q^{-}_{\rho_0})$ with
\[
\hat{\alpha} = \frac{(pq-n)r-pq}{q[(p-1)r-(p-2)]} \geq \frac{(pq-n)r-pq}{q[(r-1)(m+p-2)+1]},
\]
and
\[
\|w -w(0,0)\|_{L^{\infty}(Q^{-}_{\rho})}\leq \mathrm{K}_0\rho^{\hat{\alpha}} \quad \forall  \,\, 0<\rho<\frac{\rho_0}{2},
\]
with $\mathrm{K}_0>0$ a universal constant. Moreover, notice that
\[
\begin{array}{rcl}
  \hat{\alpha} & \geq & \min\left\{\alpha^{-}_0\max\left\{\frac{1}{\alpha_0(m+p-3)+1}, 1\right\}, \frac{(pq-n)r-pq}{q[(p-1)r-(p-2)]}\right\} \vspace{0.2cm}\\
   & \ge & \min\left\{\alpha^{-}_0\max\left\{\frac{1}{\alpha_0(m+p-3)+1}, 1\right\}, \frac{(pq-n)r-pq}{q[(r-1)(m+p-2)+1]}\right\} \\
   & = & \alpha.
\end{array}
\]

Hence, since  $\alpha \le \hat{\alpha}$, we get that
\[
\sup_{Q^-_\rho}|w(x,t) - w(0,0)| = \|w -w(0,0)\|_{L^{\infty}(Q^{-}_{\rho})} \leq \mathrm{K}_0\rho^\alpha,
\]
which implies
\begin{align*}
\sup\limits_{(x,t)\in Q^{-}_{\mu\rho}}|u(x,t)-u(0,0)| & =  \sup\limits_{(x,t)\in Q^{-}_{\rho}}|u(\mu x,\mu^\theta t)-u(0,0)| \\
& = \mu^\alpha\sup\limits_{(x,t)\in Q^{-}_{\rho}}|w(x,t)-w(0,0)| \\
& \leq \mathrm{K}_0(\mu\rho)^\alpha,
\end{align*}
for any $0<\mu\rho<\mu\frac{\rho_{0}}{2}$. Therefore
\begin{equation}\label{SecondEstThm}
\sup\limits_{(x,t)\in Q^{-}_{\rho}} \lvert u(x,t)-u(0,0) \rvert \leq \mathrm{K}_0\rho^{\alpha}, \quad  \text{for any} \quad 0<\rho<\mu\frac{\rho_{0}}{2}.
\end{equation}

\bigskip

\noindent{\it Case 3:} $\rho \in \left[\mu\frac{\rho_0}{2}, \mu\right)$: Finally, in this last case we obtain from Case 1, \eqref{Eq_rho_1}, \eqref{Eq_rho_2} and \eqref{Eq_rho_3}

{\small{
\begin{equation}\label{LastEstThm}
  \begin{array}{rcl}
\displaystyle  \sup\limits_{(x,t)\in Q^{-}_{\rho}}|u(x,t)-u(0,0)| & \le &  \displaystyle \sup\limits_{(x,t)\in Q_{\mu}} | u(x,t)-u(0,0)| \\
   & \le &  \left(\mathrm{C}+\frac{1}{4}\right)\mu^{\alpha} \\
   & \le &  \left(\mathrm{C}+\frac{1}{4}\right)\Big(\frac{2\rho}{\rho_{0}}\Big)^{\alpha} \\
   & \le & \left(\mathrm{C}+\frac{1}{4}\right)2^{\alpha} \max\left\{\left[16\gamma\left(\frac{4}{3}\right)^{\alpha_0}\right]^{\frac{\bar{p}\alpha}{\alpha_0\theta}}, \left[16\gamma\left(\frac{4}{3}\right)^{\alpha_0}\right]^{\frac{\alpha}{\alpha_0}}, \left[16\gamma\left(\frac{4}{3}\right)^{\alpha_0}\right]^{\frac{2\alpha}{\alpha_0\theta}}\right\}\rho^{\alpha}.
\end{array}
\end{equation}
}}

Now, by setting
{\small{
\[
\begin{array}{rcl}
  \mathrm{M}_0 & \defeq  & \max\left\{\mathrm{C} + \frac{1}{4}, \,\mathrm{K}_0, \, \left(\mathrm{C}+\frac{1}{4}\right)2^{\alpha} \max\left\{\left[16\gamma\left(\frac{4}{3}\right)^{\alpha_0}\right]^{\frac{\bar{p}\alpha}{\alpha_0\theta}}, \left[16\gamma\left(\frac{4}{3}\right)^{\alpha_0}\right]^{\frac{\alpha}{\alpha_0}}, \left[16\gamma\left(\frac{4}{3}\right)^{\alpha_0}\right]^{\frac{2\alpha}{\alpha_0\theta}}\right\}\right\} \\
   & = & \max\left\{\mathrm{K}_0, \, \left(\mathrm{C}+\frac{1}{4}\right)2^{\alpha} \max\left\{\left[16\gamma\left(\frac{4}{3}\right)^{\alpha_0}\right]^{\frac{\bar{p}\alpha}{\alpha_0\theta}}, \left[16\gamma\left(\frac{4}{3}\right)^{\alpha_0}\right]^{\frac{\alpha}{\alpha_0}}, \left[16\gamma\left(\frac{4}{3}\right)^{\alpha_0}\right]^{\frac{2\alpha}{\alpha_0\theta}}\right\}\right\}
\end{array}
\]
}}
we get from \eqref{FirstEstThm}, \eqref{SecondEstThm} and \eqref{LastEstThm}, the estimate \eqref{MainEstThm} for any $\rho \in (0, \lambda)$.

From a standard covering argument, we obtain
\[
[u]_{C^{\alpha, \frac{\alpha}{\theta}}\left(Q^{-}_{\frac{1}{2}}\right)} \leq \mathrm{M}_0.
\]
This finishes the proof.
\end{proof}

\section{Application to sublinear Trudinger type equations}

In this section, we apply our results to find the sharp regularity of non-negative bounded solutions for doubly nonlinear equations of the form

\begin{equation}\label{EqDNEq}
\frac{\partial (u^{\mathrm{k}})}{\partial t}-\Div(|\nabla u|^{p-2}\nabla u) = f(x, t) \,\,\, \text{in} \,\,\, Q_1^-,
\end{equation}
where
\[
\max\left\{1, \dfrac{2n}{n+2}\right\} < p \le 2 \quad \text{and} \quad  \mathrm{k} \in (0, 1).
\]
Furthermore, for the case where $p \in (2, \infty)$, see \cite[Section 6]{BJdaSR22}.

This class of equations is known as Trudinger type equations, which model turbulent filtration of non-Newtonian fluids through a porous media. One can observe that such equations have the feature of being singular in the time variable, since $u^{\mathrm{k}-1}$ blows up over the set $\{u=0\}$, and it is degenerate/singular in space, since the modulus of ellipticity, {\it i.e.,} $|\nabla u|^{p-2}$ blows up at those points where $\{|\nabla u|=0\}$. For a robust manuscript on this topic, we refer the reader to \cite{DT94} and the references therein.

In order to derive equations of type \eqref{EqDNEq} from our model \eqref{1.1}, we do a change of variables $v = u^{\mathrm{k}}$, and as long as the solutions are strictly away from zero, we obtain  that $v$ solves

\[
   \frac{\partial v}{\partial t}-\Div\left(\frac{1}{\mathrm{k}^{p-1}}v^{m_{\mathrm{k}}-1}\lvert \nabla v \rvert^{p-2}\nabla v\right) = f(x, t) \quad \text{in} \quad Q_1^-,
\]
where $m_{\mathrm{k}} := \frac{(1-\mathrm{k})(p-1)}{\mathrm{k}} +1$. In such scenario
\begin{equation}\label{EqDNEq1}
\max\left\{1, \frac{2n}{\mathrm{k}(2-n)+2n}\right\}<p \le 2\quad \text{and} \quad  \mathrm{k} \in (0, 1).
\end{equation}
Moreover, notice that by \eqref{EqDNEq1} we have $m_k \geq 1$.

Therefore, by applying Theorem \ref{t1.2}, we have that if $u$ is a weak solution to  \eqref{EqDNEq}, then $u$ is locally of class $C^{\alpha_{\mathrm{k}}, \frac{\alpha_{\mathrm{k}}}{\theta_{\mathrm{k}}}}(Q_1^-)$, where
\[
  \alpha_{\mathrm{k}} \defeq \min\left\{\alpha_0\max\left\{\frac{k}{\alpha_0(p-1-\mathrm{k}) +k}, 1 \right\},\, \,\frac{\mathrm{k}[(pq-n)r-pq]}{q[(r-1)(p-1)+\mathrm{k}]}\right\}
\]
and
\[
\theta_{\mathrm{k}} \defeq p -\alpha_{\mathrm{k}}\left(\frac{p-1}{\mathrm{k}}\right)\Big(1-\frac{\mathrm{k}}{p-1}\Big).
\]
In this case, the optimal exponent $\alpha_0>0$ can be found in \cite[Theorem 2.3]{HL13}. In addition, there exists a universal constant $\mathrm{C}_0 > 0$ such that
\[
  \displaystyle  [u]_{C^{\alpha_{\mathrm{k}}, \frac{\alpha_{\mathrm{k}}}{\theta_{\mathrm{k}}}}\left(Q_{\frac{1}{2}}^-\right)} \leq \mathrm{C}_0\left[\|u\|_{L^{\infty}(Q_1^-)} + \|f\|_{L^{q, r}(Q_1^-)}\right].
\]
For results regarding the regularity estimates to homogeneous Trundger's equation, we refer to \cite{DiUrb20, HL13} and \cite{KSU12}.

\section{Estimates for solutions close to evolution $p$-Laplacian}

In the sequel, we argue how weak solutions of
\begin{equation}\label{Eq5.1}
  \frac{\partial u}{\partial t}-\Div(m\lvert u\rvert^{m-1}\lvert \nabla u \rvert^{p-2}\nabla u) = f(x, t) \quad \text{in} \quad \quad Q_1^{-},
\end{equation}
become ``asymptotically Lipschitz continuous'' in the scenario in which our model is close, in a suitable way, to the homogeneous evolutionary $p$-Laplacian equation, {\it i.e.,}
\[
 \frac{\partial u}{\partial t} - \Delta_p u = 0 \qquad\text{in} \quad Q^{-}_1.
\]
Indeed, we follow the same lines as in \cite{Stur17-2}. Consider the parameter
\begin{equation}
\iota\coloneqq \frac{m-1}{p-1} \quad \text{for} \quad m\geq 1 \quad \text{and} \quad \max\left\{1, \frac{2n}{n+2}\right\} <p < \infty,
\end{equation}
and rewrite the equation \eqref{Eq5.1} as
\begin{equation}\label{Eq5.2}
\frac{\partial u}{\partial t}-\phi_0(m,p,\iota)\Div{(\lvert \nabla u^{\iota +1}\rvert^{p-2}\nabla u^{\iota+1})}=f(x,t) \quad \text{in} \quad Q_1^{-},
\end{equation}
where $\phi_0(m,p,\iota)=m\Big(\frac{1}{\iota +1}\Big)^{p-1}.$ Observe that $\phi_0$ does not degenerate when the parameter $m$ goes to $1$. More precisely, we have
\[
  \phi_0(m, p, \iota)\to 1 \quad  \text{as} \quad m\to 1^{+} \quad \text{for any}\,\,\,\, \max\left\{1, \frac{2n}{n+2}\right\} <p < \infty.
\]

Next, we show how weak solutions to \eqref{Eq5.1} are close, in the $L^\infty$-sense, to a $p$-harmonic function, in the case $m \to 1^+$.

\begin{lemma}[{\bf  $p$-Approximation Lemma}]\label{CaloricLemma} Let $u \in C_{\loc}(-1, 0; L^2_{\loc}(B_1))$ a weak solution to \eqref{Eq5.2} with $\|u\|_{L^{\infty}(Q^{-}_{1})}\leq 1$. Given $\delta>0$, there exist $\epsilon>0$, depending only on $n$, $p$ and $\delta$ such that if
\[
|m-1|  + \|f\|_{L^{q,r}(Q^{-}_1)} \leq \epsilon,
\]
then we can find $w$ satisfying
\[
\left\{
\begin{array}{rcrcl}
  \frac{\partial w}{\partial t}- \Delta_p w & = & 0 & \text{in} & Q^{-}_{\frac{1}{2}}\\
  w & = & u & \text{on} & \partial_p Q^{-}_{\frac{1}{2}}
\end{array}
\right.
\]
such that
\[
  \sup\limits_{Q^{-}_{\frac{1}{2}}}\lvert w - u\rvert \leq \delta.
\]
\end{lemma}

\begin{proof} The proof follows the same lines as the one in \cite[Lemma 6.1]{BJdaSR22}. For this reason, we will omit it here.
\end{proof}

Finally, we can show the asymptotic regularity estimates, in the case where we have $m \to 1^+$.

\begin{theorem}\label{Mthm2}\label{ThmAsymReg} Let $u \in C_{\loc}(-1, 0; L^2_{\loc}(B_1))$ be a bounded weak solution of \eqref{Eq5.1} with $f \in L^{q, r}(Q^{-}_1)$. Suppose that A\ref{assump_w-cc} holds true. Given, $\alpha \in(0,1)$, there exists an $\varepsilon_0>0$ such that if $m-1 < \varepsilon_0$, then any solution of \eqref{Eq5.1} belongs to $C^{\alpha, \frac{\alpha}{\theta}}$. Moreover, there exists a universal constant $\mathrm{M}_0>0$ such that
$$
  \displaystyle  [u]_{C^{\alpha, \frac{\alpha}{\theta}}\left(Q^{-}_{\frac{1}{2}}\right)} \leq \mathrm{M}_0\left[\|u\|_{L^{\infty}(Q^{-}_1)} + \|f\|_{L^{q, r}(Q^{-}_1)}\right].
$$
Quantitatively, such estimate states that $u \in C^{1^{-}, {\frac{1}{2}}^{-}}\left(Q^{-}_{\frac{1}{2}}\right)$.
\end{theorem}

\begin{proof} Combining the $p$-Approximation Lemma \ref{CaloricLemma} with $C_{\text{loc}}^{1, {\frac{1}{2}}}$-regularity estimates available for the homogeneous profiles $\frac{\partial \mathfrak{h}}{\partial t}- \Delta_p \mathfrak{h}  =  0$ (cf. \cite{AMS04} and \cite{D93}), we can proceed similarly as in the proof of Theorem \ref{t1.2}.
\end{proof}

\subsection*{Acknowledgments}

\hspace{0.65cm} P. Andrade was partially supported by the Portuguese government through FCT-Fundação para a Ciência e a Tecnologia, I.P., under the project UID/MAT/04459/2020. J.V. da Silva was partially supported by Conselho Nacional de Desenvolvimento Cient\'{i}fico e Tecnol\'{o}gico (CNPq-Brazil) under Grants No. 307131/2022-0 and FAPDF Demanda Espont\^{a}nea 2021 and   FAPDF - Edital 09/2022 - DEMANDA ESPONTÂNEA. M. Santos was partially supported by the Portuguese government through FCT-Fundação para a Ciência e a Tecnologia, I.P., under the projects UID/MAT/04459/2020 and PTDC/MAT-PUR/1788/2020.


\end{document}